\documentclass[12pt]{amsart}

\usepackage{amssymb}
\usepackage{graphicx,color}
\usepackage{cleveref}

\newtheorem{theorem}{Theorem}[section]

\newtheorem{lemma}[theorem]{Lemma}

\theoremstyle{definition}
\newtheorem{definition}[theorem]{Definition}

\newtheorem{example}[theorem]{Example}

\newtheorem{question}[theorem]{Question}

\newtheorem{remark}[theorem]{Remark}

\newcommand{\FF}{ \ensuremath{\mathbb{F}}}

\makeatletter
\def\moverlay{\mathpalette\mov@rlay}
\def\mov@rlay#1#2{\leavevmode\vtop{%
   \baselineskip\z@skip \lineskiplimit-\maxdimen
   \ialign{\hfil$\m@th#1##$\hfil\cr#2\crcr}}}
\newcommand{\charfusion}[3][\mathord]{
    #1{\ifx#1\mathop\vphantom{#2}\fi
        \mathpalette\mov@rlay{#2\cr#3}
      }
    \ifx#1\mathop\expandafter\displaylimits\fi}
\makeatother
\newcommand{\bigcupdot}{\charfusion[\mathop]{\bigcup}{\cdot}}

\newcommand{\lk}{{\mathrm{lk}}}

\begin{document}

\title{A balanced non-partitionable Cohen-Macaulay complex}

\author[M. Juhnke-Kubitzke]{Martina Juhnke-Kubitzke}
\email{juhnke-kubitzke@uni-osnabrueck.de}
\author[L. Venturello]{Lorenzo Venturello}
\email{lorenzo.venturello@uni-osnabrueck.de}
\address{
Universit\"{a}t Osnabr\"{u}ck,
Fakult\"{a}t f\"{u}r Mathematik,
Albrechtstra\ss e 28a,
49076 Osnabr\"{u}ck, GERMANY
}

\date{\today}

\thanks{
Both authors were supported by the German Research Council DFG GRK-1916.
}
%Keyword and Subject Classes (if needed)

\begin{abstract}
In a recent paper, Duval, Goeckner, Klivans and Martin disproved the longstanding conjecture by Stanley, that every Cohen-Macaulay simplicial complex is  partitionable. We construct counterexamples to this conjecture that are even \emph{balanced}, i.e., their underlying graph has a minimal coloring. This answers a question by Duval et al. in the negative.
\end{abstract}

\maketitle

\tableofcontents

\section{Introduction}
Undoubtedly, Cohen-Macaulay simplicial complexes are among the best studied classes of simplicial complexes in topological combinatorics and combinatorial commutative algebra and they have been proven to be extremely useful for various problems in these areas. The most prominent such example is probably provided by Stanley's proof of the Upper Bound Conjecture for spheres, which also marks the birth of Cohen-Macaulay complexes \cite{Stanley-UBC}. Though the original definition by Stanley is algebraic, Reisner \cite{Reisner}~--~using results of Hochster \cite{Hochster}~--~could show that Cohen-Macaulayness is a purely topological property. In particular, all triangulations of balls and spheres are known to be Cohen-Macaulay. 
One of the longstanding conjectures concerning this class of simplicial complexes is the so-called \emph{Partitionability Conjecture} by Stanley \cite[p. 14]{St79} (for all Cohen-Macaulay simplicial complexes) and Garsia \cite{GAR} (for barycentric subdivisions), stating that every Cohen-Macaulay simplicial complex is partitionable. An affirmative answer to this conjecture would also have provided a combinatorial interpretation of the $h$-vectors of Cohen-Macaulay complexes. However, last year, Duval, Goeckner, Klivans and Martin \cite{DGKM} provided an infinite family of non-partitionable Cohen-Macaulay simplicial complexes~--~together with a general construction method for such counterexamples~--, and thereby disproved Stanley's conjecture. Even though we now know that the Partitionability Conjecture is false in full generality, one could still hope for a more restricted version to be true. In particular, Duval et al. suggested the following question, which is the main focus of this article \cite[Question 4.2]{DGKM}.

\begin{question}\label{mainquestion}
Is every balanced Cohen-Macaulay simplicial complex partitionable?
\end{question}

We recall that a $(d-1)$-dimensional simplicial complex is called \emph{balanced} if its underlying graph is $d$-colorable (in the graph-theoretic sense). Balanced simplicial complexes were introduced by Stanley \cite{St79} and they comprise Coxeter complexes, Tits buildings and also barycentric subdivisions of regular CW complexes. Hence, a counterexample to the Partitionability Conjecture for barycentric subdivisions, as proposed by Garsia \cite{GAR}, would also answer \Cref{mainquestion} in the negative.

Using the technique introduced in \cite{DGKM}, we construct an infinite family of balanced non-partitionable Cohen-Macaulay complexes. The main idea is to start with the counterexample from \cite{DGKM} and to \emph{remove} the obvious obstructions to balancedness. We will make this idea more precise in \Cref{section:construction}. Indeed, our counterexample can be obtained from the one in \cite{DGKM} by performing a finite number of edge subdivisions. As in \cite{DGKM}, our counterexample is not only Cohen-Macaulay but even constructible. As such, it is the first example of a balanced constructible non-partitionable simplicial complex \cite[\S 4]{Hachimori}. 

The article is structured as follows. In \Cref{sect:Combinatorics}, we introduce the necessary background on the combinatorics of simplicial complexes and recall the construction of the counterexample from \cite{DGKM}. Building on this, \Cref{section:construction} contains the construction of our counterexample (see Theorems \ref{mainresult1} and \ref{mainresult2}).

\section{Background on simplicial complexes}\label{sect:Combinatorics}
We recall basics on (relative) simplicial complexes, including some of their combinatorial and algebraic properties. We refer to \cite{BH-book} and \cite{Stanley-greenBook} for more details.

Given a finite set $V$, an (abstract) \emph{simplicial complex} $\Delta$ on the vertex set $V$ is a collection of subsets of $V$ that is closed under inclusion. In the following, we will write $V(\Delta)$ for the vertex set of a simplicial complex $\Delta$. 
Throughout this paper, all simplicial complexes are assumed to be finite. 
Elements of $\Delta$ are called \emph{faces} of $\Delta$ 
and inclusion-maximal faces are called \emph{facets} of $\Delta$. The \emph{dimension} of a face $F \in \Delta$ is its cardinality minus one, and the \emph{dimension} of $\Delta$ is defined as $\dim\Delta:=\max\{\dim F~:~F\in \Delta\}$. $0$-dimensional and $1$-dimensional faces are called \emph{vertices} and \emph{edges}, respectively. A simplicial complex $\Delta$ is \emph{pure} if all its facets have the same dimension. The \emph{link} of a face $F\in \Delta$ is the subcomplex
$$
\lk_{\Delta}(F)=\{G\in\Delta~:~G\cap F=\emptyset,\; G\cup F\in \Delta\}.
$$
A subcomplex $\Gamma\subseteq \Delta$ is \emph{induced} if for any $F\subseteq V(\Gamma)$ with $F\in \Delta$, it holds that $F\in \Gamma$. We write $\Gamma=\Delta_{V(\Gamma)}$ in this case. 
A \emph{relative simplicial complex} is a pair $(\Delta,\Gamma)$ of simplicial complexes, where $\Gamma\subseteq \Delta$ is a subcomplex of $\Delta$. As for usual simplicial complexes,  elements of $\Delta\setminus \Gamma$ are called \emph{faces} of $(\Delta,\Gamma)$ and the \emph{dimension} of $(\Delta,\Gamma)$ is the maximal dimension of a face in $\Delta\setminus\Gamma$. Other notions from arbitrary simplicial complexes carry over to relative simplicial complexes in exactly the same way. If $(\Delta,\Gamma)$ is a relative simplicial complex and $\Omega$ is a simplicial complex, then the pair $(\Delta\cup\Omega,\Gamma\cup \Omega)$ represents the same relative simplicial complex as $(\Delta,\Gamma)$ and every pair of simplicial complexes representing $(\Delta,\Gamma)$ arises in this way. In particular, every relative simplicial complex $(\Delta,\Gamma)$ has a unique minimal representation $(\overline{\Omega},\overline{\Omega}\setminus \Omega)$, where $\Omega=\Delta\setminus \Gamma$ and 
$$
\overline{\Omega}=\{F~:~F\subseteq G \mbox{ for some } G\in \Omega\}
$$ 
is the minimal simplicial complex containing $\Omega$. We also call $\overline{\Omega}$ the \emph{combinatorial closure} of $\Omega$. We will make use of this minimal representation of relative simplicial complexes in the construction of our counterexample. 

The \emph{$f$-vector} of a $(d-1)$-dimensional (relative) simplicial complex $\Delta$ is $f(\Delta)=(f_{-1}(\Delta), f_0(\Delta), \ldots, f_{d-1}(\Delta))$, where $f_i(\Delta)$ denotes the number of $i$-dimensional faces of $\Delta$. The \emph{$h$-vector} $h(\Delta)=(h_0(\Delta),h_1(\Delta),\ldots,h_d(\Delta))$ of $\Delta$ is defined by the relation
$$
\sum_{i=0}^{d}f_{i-1}(\Delta)(t-1)^{d-i}=\sum_{i=0}^dh_i(\Delta)t^{d-i}.
$$

We now turn to several combinatorial and algebraic properties of simplicial complexes that will be of importance for this article.

A $(d-1$)-dimensional simplicial complex $\Delta$ is called \emph{balanced} if its underlying graph is $d$-colorable, that is, there exists a map $\kappa:\;V(\Delta)\to [d]=\{1,\ldots,d\}$ such that $\kappa(v)\not=\kappa(w)$ if $\{v, w\}\in \Delta$. Balanced simplicial complexes were originally introduced by Stanley \cite{St79} and prominent examples of such simplicial complexes are provided by barycentric subdivisions, Coxeter complexes and Tits buildings. 

We now recall 
the definition of partitionability which goes back to Ball \cite{Ball} and Provan \cite{Provan}. 
\begin{definition}
A pure (relative) simplicial complex $\Delta$ with facets $F_1,\ldots,F_n$ is called \emph{partitionable} if there exists a partitioning of $\Delta$ into pairwise disjoint Boolean intervals
$$
\Delta=\bigcupdot_{i=1}^n[R_i,F_i],
$$
where $[R_i,F_i]=\{G\in \Delta~:~R_i\subseteq G\subseteq F_i\}$. 
\end{definition}
It was shown by Stanley \cite[Proposition III.2.3]{Stanley-greenBook} that the $h$-vector of a partitionable simplicial complex $\Delta$ has the following combinatorial interpretation: 
\begin{equation}\label{eq:partition}
h_i(\Delta)=\#\{1\leq j\leq n~:~\#R_j=i\}.
\end{equation}
 In particular, all $h$-vector entries are non-negative in this case.  

We now define three classes of simplicial complexes that are known to share the same set of $h$-vectors \cite[Theorem 6]{Stanley77}: shellable, constructible and Cohen-Macaulay simplicial complexes. 

A pure simplicial complex $\Delta$ is \emph{shellable} if there exists an ordering $F_1,\ldots,F_n$ of the facets of $\Delta$ such that for each $1\leq i\leq n$ there exists a unique minimal element $R_i$ in 
$$\{G\subseteq F_i~:~G\not\subseteq F_j\mbox{ for } 1\leq j\leq i-1\}.$$
Since, obviously, $\bigcupdot_{i=1}^n[R_i,F_i]$ is a partitioning of $\Delta$, the $h$-vector of a shellable simplicial complex can be computed using \eqref{eq:partition}.

A $(d-1)$-dimensional simplicial complex $\Delta$ is \emph{constructible} if $\Delta$ is a simplex or $\Delta=\Delta_1\cup\Delta_2$, where $\Delta_1$ and $\Delta_2$ are constructible $(d-1)$-dimensional simplicial complexes and $\Delta_1\cap\Delta_2$ is constructible of dimension $d-2$. 

In the following, let $\FF$ be an arbitrary field. A $(d-1)$-dimensional simplicial complex  $\Delta$ is \emph{Cohen-Macaulay} over $\FF$ if and only if, for every face $F\in \Delta$ (including the empty face), $\tilde{\beta}_i(\lk_\Delta(F);\FF)=0$ for all $i \ne d-1-\#F$ (see \cite[Corollary II.4.2]{Stanley-greenBook}). Here, we use $\tilde{\beta}_i(\Gamma;\FF):=\dim_{\FF}\widetilde{H}_i(\Gamma;\FF)$ to denote the dimension of the $i$\textsuperscript{th}
 reduced homology group of a simplicial complex $\Gamma$ over $\FF$. 

The following implications between the just described classes of simplicial complexes are well-known:
 \begin{align*}
&\mbox{shellable} \qquad \Rightarrow \qquad \mbox{constructible} \qquad \Rightarrow\qquad \mbox{Cohen-Macaulay}\\
&\qquad\Downarrow\\
&\mbox {partitionable}.
 \end{align*}
Examples of constructible 3-balls that are not shellable due to Rudin \cite{Rudin} and Ziegler \cite{Ziegler} show that the left implication is strict. Moreover, any triangulation of the dunce hat is known to be Cohen-Macaulay but not constructible \cite[\S 2]{Hachimori08}, which means that also the right implication is strict. It is natural to ask how partitionability is related to constructibility and Cohen-Macaulayness. While already more than $30$ years ago Bj\"orner constructed a partitionable simplicial complex that is not Cohen-Macaulay and hence neither constructible nor shellable \cite[p. 85]{Stanley-greenBook}, the \emph{Partitionability Conjecture} due to  Stanley \cite{St79} (for all Cohen-Macaulay complexes) and Garsia (for barycentric subdivisions) stated that Cohen-Macaulayness implies partitionability. It was shown only last year by Duval, Goeckner, Klivans and Martin \cite{DGKM} that there do also exist constructible, hence Cohen-Macaulay simplicial complexes that are not partitionable. 
Since our construction of a balanced non-partitionable Cohen-Macaulay complex, is essentially a balanced version of the  example from \cite{DGKM}, we now recall this construction.

\begin{example}\label{example:DGKM}
The construction in \cite{DGKM} starts with a particular subcomplex $Q$ of Ziegler's famous example of a non-shellable $3$-ball on $10$ vertices, labeled $0,..,9$ \cite{Ziegler}. More precisely, the subcomplex $Q$ is the combinatorial closure of the following set of facets
\begin{align*}
\mathcal{F}(Q)=\{&\{1,2,4,9\}, \{1,2,6,9\},\{1,5,6,9\},\{1,5,8,9\},\{1,4,8,9\},\\
&\{1,4,5,8\},\{1,4,5,7\},\{4,5,7,8\},
\{1,2,5,6\},\{0,1,2,5\},\\
&\{0,2,5,6\},\{0,1,2,3\},\{1,2,3,4\},\{1,3,4,7\}\}.
\end{align*}
Let $A=Q_{\{0,2,3,4,6,7,8\}}$ be the induced subcomplex of $Q$ on vertex set $\{0,2,3,4,6,7,8\}$, i.e., $A$ is the combinatorial closure of 
\begin{equation*}
\{\{0,2,6\},\{0,2,3\},\{2,3,4\},\{3,4,7\},\{4,7,8\}\}.
\end{equation*}
The complexes $A$ and $Q$ are depicted in \Cref{qbar}. 
Theorem 3.4 in \cite{DGKM} shows that glueing together $25$ copies of $Q$ along the subcomplex $A$ produces a non-partitionable, constructible, hence Cohen-Macaulay simplicial complex. In fact, Theorem 3.5 in \cite{DGKM} shows that already $3$ copies of $Q$, identified along $A$, yield such an example. 
However, since $Q$ is not balanced, neither are those examples. One important fact, that was used intensively in \cite{DGKM} and that we will also employ in the next section, is that $\tau=(0\,7)(2\,4)(6\,8)$ is an automorphism of $Q$. 
\end{example}
\begin{figure}[h]
	\centering
	\includegraphics[scale=0.8]{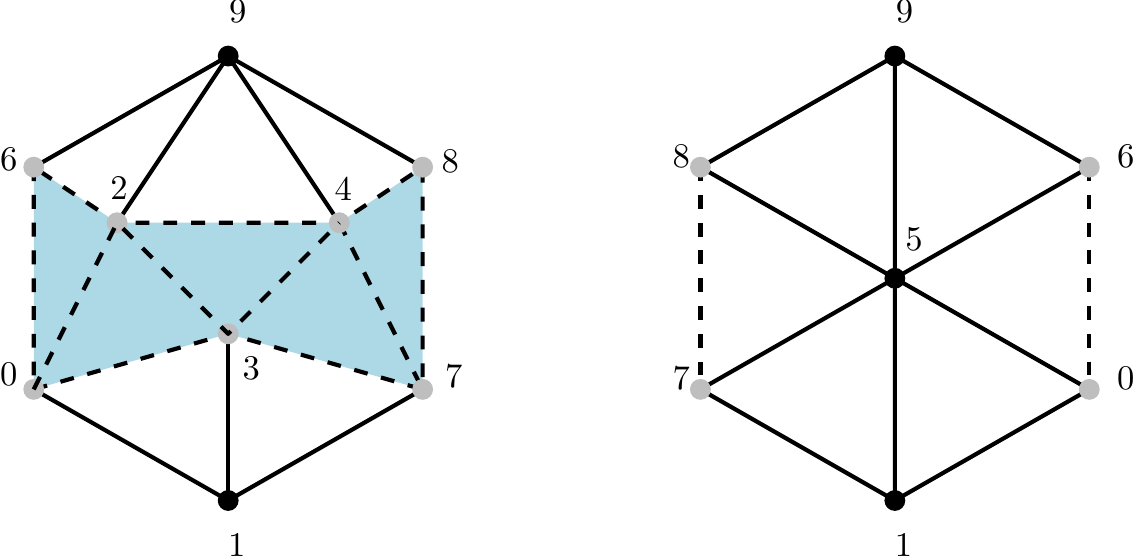}
	\caption{Front (\emph{left}) and back (\emph{right}) view of $Q$. The blue and dashed faces belong to $A$.}
	\label{qbar}
\end{figure}

\section{The construction}\label{section:construction}
In this section, we provide our construction of a balanced non-partitionable Cohen-Macaulay simplicial complex.  

The following main tool from \cite{DGKM} is crucial for this construction.

\begin{theorem}
\cite[Theorem 3.1]{DGKM} \label{keythm}
	Let $X=\left(Q,A\right)$ be a relative simplicial complex such that:
	\begin{enumerate}
		\item $Q$ and $A$ are Cohen-Macaulay;
		\item $A$ is an induced subcomplex of $Q$ of codimension at most 1;
		\item $X$ is not partitionable. 
	\end{enumerate} 
	For a positive integer $N$ let $C_N$ denote the simplicial complex, obtained by identifying $N$ disjoint copies of $Q$ along $A$. If $N$ is larger than the total number of faces of $A$, then $C_N$ is Cohen-Macaulay and not partitionable.
\end{theorem}

Our construction makes use of the following basic observation.

\begin{lemma}\label{lem:vertexLinks}
Let $d-1\geq 2$ and let $\Delta$ be a balanced $(d-1)$-dimensional simplicial complex. Then all vertex links of $\Delta$ are balanced.
\end{lemma}

\begin{proof}
The claim follows immediately from the facts that $\Delta$ is balanced, \\
$\dim\lk_\Delta(\{v\})=d-2$ for any vertex $v\in V(\Delta)$ and that $\kappa(v)\neq \kappa(w)$ for all $v\in V(\Delta)$, $w\in \lk_\Delta(\{v\})$ and any proper coloring $\kappa:\;V(\Delta)\to \{1,\ldots,d\}$.
\end{proof}

Note that the converse of the previous lemma is not necessarily true. In the following, we call a vertex $v$ (or its vertex link $\lk_\Delta(\{v\})$)  \emph{critical} if $\lk_\Delta(\{v\})$ is not balanced and \emph{uncritical} otherwise. 

Before proceeding to our construction, we describe its underlying idea. 
If we look at the simplicial complex $Q$ from \Cref{example:DGKM}, we easily see that vertices $0,3,7$ are uncritical, whereas all the other vertices are critical. Hence, by the previous lemma, those are obvious obstructions that prevent $Q$ from being balanced. The idea now is to perform some (possibly few) edge subdivisions that make the critical vertex links balanced without affecting balancedness of other uncritical vertex links and without altering the symmetry of the simplicial complex $Q$. Luckily,~--~though this is not guaranteed by \Cref{lem:vertexLinks}~--~it will turn out, that the simplicial complex, obtained in this way, is already balanced. 
We now make this idea more precise. We perform the following subdivision steps:\\

\begin{itemize}
\item[{\sf Step 1:}] We first subdivide the edge $\{2,4\}$ by introducing a new vertex $10$. In this way, the link of the former critical vertex $9$ becomes the cone over a $6$-gon and as such is balanced. \Cref{lks} shows the link of the vertex $9$ before and after the subdivision. Moreover, the permutation $\tau=(0\,7)(2\,4)(6\,8)$ is still an automorphism of the subdivided complex (see \Cref{example:DGKM}). \\
\begin{figure}[h]
	\centering
	\includegraphics[scale=0.6]{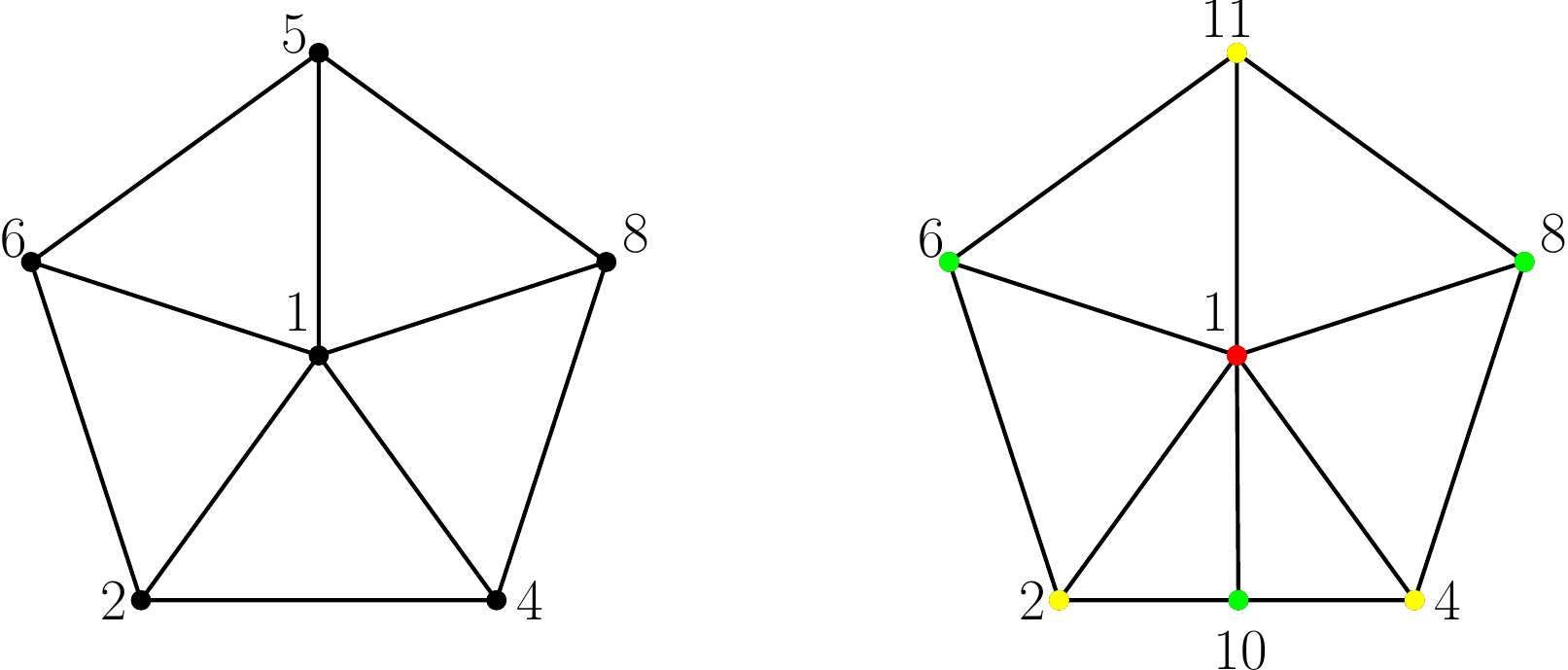}
	\caption{The link of $9$ before (\emph{left}) and after (\emph{right}) the edge subdivision of $\{2,4\}$ and $\{5,9\}$. On the right the vertices are properly colored.}
	\label{lks}
\end{figure} 
\item[{\sf Step 2:}] In the next step, we subdivide the edge $\{5,9\}$ by adding a vertex $11$. It is easy to check that the vertices $6$ and $8$ are now uncritical and that $\tau$ is still an automorphism of this new complex.\\

\item[{\sf Step 3:}] Subdividing the edges $\{0,6\}$ and $\{7,8\}$, the vertices $2$ and $4$, respectively become uncritical. Moreover, also $5$ has an uncritical link now. Labeling the new vertex on the edge $\{0,6\}$ with $12$ and the one on $\{7,8\}$ with $13$, we also see that the permutation $\tau'=(0\,7)(2\,4)(6\,8)(12\,13)$ is an automorphism of the subdivided complex.\\
\end{itemize}
We call the simplicial complex, obtained from the just described edge subdivisions $Q^*$.  The top row of \Cref{delta_balanced} depicts the front and back view of $Q^*$. It is easy to check that~--~though we did not treat the critical vertex $1$ separately~--~ the simplicial complex $Q^*$ has only uncritical vertices. We even have the following:

\begin{lemma}\label{lem:balanced}
The simplicial complex $Q^*$ constructed above is balanced.
\end{lemma}

\begin{proof}
The bottom row of \Cref{delta_balanced} shows the $3$-dimensional simplicial complex $Q^*$ together with a proper $4$-coloring.
\end{proof}
 \begin{figure}[h]
 	\centering
 	\includegraphics[scale=0.8]{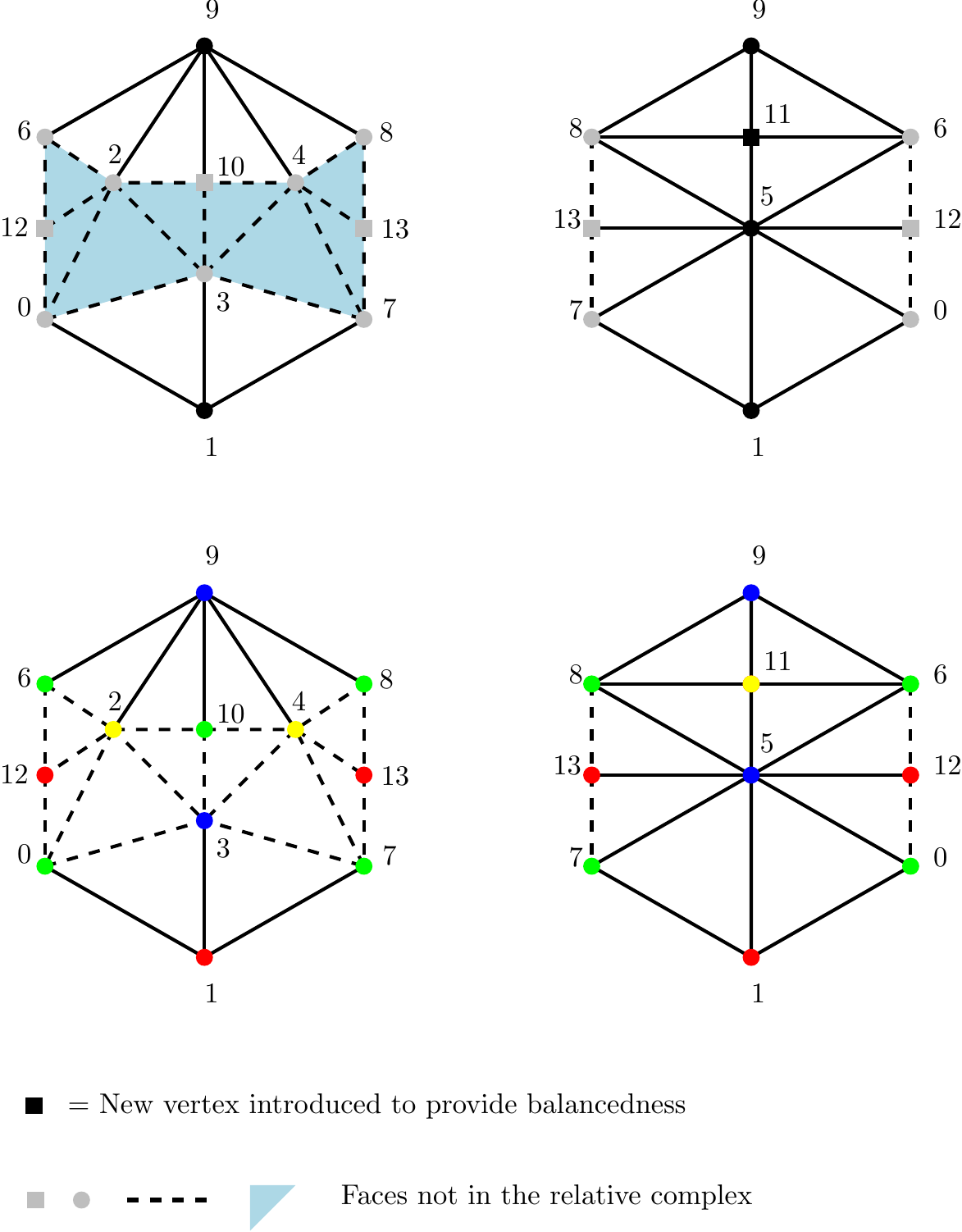}
 	\caption{Front and back view of $\Delta$. }
 	\label{delta_balanced}
 \end{figure} 
 Another reasonable approach to construct a balanced counterexample to the partitionability conjecture could have been to start with a balanced non-shellable ball and then to try to apply the technique from \cite{DGKM}. However, all examples of balanced non-shellable balls are relatively big and it is hard to see, which subcomplex one should choose then. 
 We also want to remark that applying the same edge subdivisions as above directly to Ziegler's non-shellable ball, does not produce a balanced ball. 
 
The following simple remark will be useful later. 

\begin{remark}\label{rem:simple}
If $\Delta$ is a balanced simplicial complex, then any simplicial complex built from $\Delta$ by taking a certain number of copies of $\Delta$ and identifying them along a fixed subcomplex, is balanced.
\end{remark}

We define $A^*$ of $Q^*$ to be the induced subcomplex $Q^*_{\{0,2,3,4,6,7,8,10,12,13\}}$. Note that $\dim A^*=2$ and that $A^*$ can be obtained from $A$ in the same way we constructed $Q^*$ from $Q$; namely, by subdividing the edges $\{2,4\}$, $\{0,6\}$ and $\{7,8\}$. (We do not subdivide $\{5,9\}$ since it is not present in $A$.) As edge subdivisions preserve shellability, we get the following lemma:

\begin{lemma}\label{lem:constructible}
The simplicial complexes $Q^*$ and $A^*$ are shellable, hence constructible and Cohen-Macaulay. 
\end{lemma}

Our final goal is to apply \Cref{keythm} to the relative simplicial complex $(Q^*,A^*)$. The only ingredient missing to be able to do so is to verify that condition (3) of \Cref{keythm} is fulfilled. Indeed, we have the following statement:

\begin{theorem}\label{thm:partitionable}
The relative simplicial complex $X:=(Q^*,A^*)$ is not partitionable.
\end{theorem}

\begin{proof}
The proof uses similar ideas as the ones employed in the proof of \cite[Theorem 3.3]{DGKM}. In fact, some parts are even verbatim the same but we include them for sake of completeness. 

Assume by contradiction that $X$ is partitionable. We will show that, in this case, the vertex $5$ has to be contained in at least two intervals of any partitioning, which gives a contradiction.\\
For the sake of clearness we list the facets of both $Q^*$ and $A^*$.
 	\begin{align*}
 	\mathcal{F}\left(Q^*\right)=\big\{&\{2, 5, 6, 12\}, \{1, 5, 8, 11\}, \{0, 1, 2, 5\},\{1, 4, 8, 9\}, \{1, 2, 3, 10\}, \{4, 5, 8, 13\}, \\
 	&\{0, 1, 2, 3\}, \{1, 4, 9, 10\},\{1, 2, 9, 10\}, \{1, 2, 5, 6\}, \{0, 2, 5, 12\}, \{4, 5, 7, 13\},\\
 	& \{1, 4, 5, 8\}, \{1, 2, 6, 9\}, \{1, 3, 4, 7\}, \{1, 3, 4, 10\},\{1, 5, 6, 11\}, \{1, 4, 5, 7\},\\
 	& \{1, 6, 9, 11\}, \{1, 8, 9, 11\}\big\}\\
\mathcal{F}\left( A^*\right)=&\big\{\{ 0, 2, 3\} , \{ 4, 7, 13\} , \{ 3, 4, 10\} , \{0, 2, 12\} , \{3, 4, 7\} , \{2, 3, 10\}, \{2, 6, 12\} ,\\
& \{ 4, 8, 13\}\big\} 
 	\end{align*}
Given a partitioning $\mathcal{P}$ of $X$ and a facet $F\in\mathcal{F}(Q^*)$, we denote by $I_F$ the interval of $\mathcal{P}$ with top element $F$. \\
As $\{1,4,8,9\}$ is the only facet containing the triangle $\{4,8,9\}$, we have $\{4,8,9\}\in I_{\{1,4,8,9\}}$. If also $\{1,4,8\}\in I_{\{1,4,8,9\}}$, it follows that $I_{\{1,4,8,9\}}$ must contain $\{4,8\}=\{4,8,9\}\cap\{1,4,8\}$, which is a contradiction since $\{4,8\}\in A^*$. Therefore, $\{1,4,8\}\notin I_{\{1,4,8,9\}}$ and since $\{1,4,5,8\}$ is the only other facet of $Q^*$ containing $\{1,4,8\}$, we conclude that $\{1,4,8\}\in I_{\{1,4,5,8\}}$. Again, as $\{4,8\}\in A^*$ and $\{4,8\}=\{1,4,8\}\cap\{4,5,8\}$, it must hold that $\{4,5,8\}\notin I_{\{1,4,5,8\}}$ and hence also $\{4,5\}\notin I_{\{1,4,5,8\}}$. The other facets of $X$ containing $\{4,5\}$ are
\begin{equation}\label{eq:457}
\{4,5,8,13\},\{4,5,7,13\},\{1,4,5,7\}.
\end{equation}
Using that $\tau'$ is an automorphism of $Q^*$ and $A^*$, the same line of arguments applied to $\{2,6,9\}$ yields that the edge $\{2,5\}$ has to be contained in an interval with one of the following top elements:
\begin{equation*}
\{2,5,6,12\},\{0,2,5,12\},\{0,1,2,5\}.
\end{equation*}

We now distinguish four cases:\\
{\sf Case 1:} $\{4,5\}\in I_{\{4,5,8,13\}}$ and $\{2,5\}\in I_{\{2,5,6,12\}}$\\
As $\{4,5,8,13\}$ is the only facet containing $\{5,8,13\}$, we must have $\{5,8,13\}\in \{4,5,8,13\}$. As $\{5\}=\{4,5\}\cap\{5,8,13\}$ we infer that that $\{5\}\in I_{\{4,5,8,13\}}$. Similarly, using again that $\tau'$ is an automorphism of $Q^*$ and $A^*$, we get that $\{5\}\in I_{\{2,5,6,12\}}$. Hence $\{5\}$ is contained in two intervals, which is a contradiction.\\

\noindent {\sf Case 2:} $\{4,5\}\notin I_{\{4,5,8,13\}}$ and $\{2,5\}\notin I_{\{2,5,6,12\}}$\\
As $\{4,5\}\notin I_{\{4,5,8,13\}}$ it follows from \eqref{eq:457} that $\{4,5\}\in I_{\{4,5,7,13\}}$ or $\{4,5\}\in I_{\{1,4,5,7\}}$. Since $\{4,5,7,13\}$ and $\{1,4,5,7\}$ are also the only facets of $Q^*$ containing $\{4,5,7\}$ and since $\{4,5\},\{5,7\}\subseteq \{4,5,7\}$, it follows that they have to lie in the same interval, together with $\{5\}=\{4,5\}\cap\{5,7\}$. Therefore, we either have 
\begin{equation}\label{eq:case2}
\{5\}\in I_{\{1,4,5,7\}}\qquad \mbox{or}\qquad \{5\}\in I_{\{4,5,7,13\}}.
\end{equation}
Applying the automorphism $\tau'$ to the above argument yields 
\begin{equation*}
\{5\}\in I_{\{0,1,2,5\}}\qquad \mbox{or}\qquad \{5\}\in I_{\{0,2,5,12\}}.
\end{equation*}
Hence, $\{5\}$ belongs to two intervals, which is a contradiction.\\

\noindent {\sf Case 3:} $\{4,5\}\notin I_{\{4,5,8,13\}}$ and $\{2,5\}\in I_{\{2,5,6,12\}}$\\
 As $\{4,5\}\notin I_{\{4,5,8,13\}}$, the argument of Case 2 shows that \eqref{eq:case2} holds. We now show that $\{5\}$ has to lie in a second interval. \\
Note that the only two facets containing $\{5,12\}$ are $\{2,5,6,12\}$ and $\{0,2,5,12\}$. Since both of these contain $\{2,5,12\}$ and $\{5,12\}\subseteq \{2,5,12\}$, it follows that $\{5,12\}$ and $\{2,5,12\}$ have to belong to the same interval. Moreover, since $\{2,5\}\in I_{\{2,5,6,12\}}$ by assumption, we must have $\{2,5,12\}\in I_{\{2,5,6,12\}}$ and hence $\{5,12\}\in I_{\{2,5,6,12\}}$. Finally, this implies $\{5\}=\{2,5\}\cap \{5,12\}\in I_{\{2,5,6,12\}}$ and therefore again $\{5\}$ lies in two intervals, which is a contradiction.\\

\noindent{\sf Case 4:} $\{4,5\}\in I_{\{4,5,8,13\}}$ and $\{2,5\}\notin I_{\{2,5,6,12\}}$\\
We reach a contradiction in this case by applying the automorphism $\tau'$ to the arguments of Case 3. This finishes the proof.
\end{proof}

The (relative) simplicial complexes $Q^*,A^*$ and $X=(Q^*,A^*)$ have the following $f$-vectors:
 \begin{align*}
 &f(Q^*)=(1,14,45,52,20)\\
 &f(A^*)=(1,10,17,8)\\
 &f(X)=(0,4,28,44,20).
 \end{align*}
 In particular, the subcomplex $A^*$ has a total number of $36$ faces. \Cref{keythm},  \Cref{thm:partitionable}, \Cref{lem:constructible} and \Cref{rem:simple} therefore imply our main result:
 
 \begin{theorem}\label{mainresult1}
 	The simplicial complex $C_{37}$ constructed from $37$ disjoint copies of $Q^*$ and identifying them along $A^*$ is balanced, Cohen-Macaulay and not partitionable. 
 \end{theorem}

Analogous to the situation in \cite[Theorem 3.5]{DGKM} we note that a much smaller counterexample to the balanced partitionability conjecture can be found by glueing together only 3 copies of $Q^*$.

 \begin{theorem}\label{mainresult2}
 The simplicial complex $C_{3}$ obtained by taking 3 disjoint copies of $Q^*$ and identifying them along $A^*$ is balanced, Cohen-Macaulay and not partitionable. 
 \end{theorem}	
We omit the proof of the above theorem since it is verbatim the same as the one of Theorem 3.5 in \cite{DGKM}, if one exchanges the automorphism $\tau$ by $\tau'$.

The $f$-vector of the simplicial complex  $C_3$ is $f(C_3)=(1,22,101,140,60)$.

\begin{remark}
We do not know if $C_3$ is the smallest balanced simplicial complex that is Cohen-Macaulay but not partitionable. However, it is easy to see, e.g., by solving a certain integer linear programm, that $C_2$ is partitionable. Moreover, it is not possible to change $Q$ into  a balanced simplicial complex by fewer than $4$ edge subdivisions. On the other hand, if one finds a counterexample to the partitionability conjecture, which is smaller than the one of \cite{DGKM}, then it might well be the case that one can also construct a counterexample to the balanced partitionability conjecture that is smaller than $C_3$.
\end{remark}

\begin{remark}
As $Q^*$ and $A^*$ are both constructible by \Cref{lem:constructible}, it follows by definition that $C_3$ is also constructible. The simplicial complex $C_3$ therefore is the first balanced counterexample to the conjecture that every constructible simplicial complex is partitionable \cite[\S 4]{Hachimori}.
\end{remark}

\begin{remark}
We still do not know if every barycentric subdivision of a Cohen-Macaulay simplicial complex is partitionable \cite{GAR}. If this is not the case, then a corresponding counterexample has to be the barycentric subdivision of a non-partitionable Cohen-Macaulay simplicial complex.  Xuan Thanh Le helped us to verify that the barycentric subdivision of the simplicial complex $C_3$ from \cite{DGKM} is partitionable.
\end{remark}

\section*{Acknowledgement}
The authors would like to thank Xuan Thanh Le for his help in solving certain integer linear programms efficiently, while searching for a counterexample.

\bibliographystyle{alpha}

\bibliography{bibliography}

\end{document}